\documentclass[colorlinks]{article}
\pdfoutput=1
\usepackage{graphicx}

\usepackage[utf8]{inputenc}
\usepackage[english]{babel}

\usepackage{amsmath}
\usepackage{amsfonts}
\usepackage{amssymb}
\usepackage{amsthm}

\usepackage[all]{xy}

\usepackage[inline]{enumitem}
\setlist{itemsep=1pt,topsep=3pt,partopsep=0pt,leftmargin=20pt}
\SetEnumitemKey{(1)}{label=\rm{(\arabic*)}}
\SetEnumitemKey{(i)}{label=\rm{(\roman*)}}
\SetEnumitemKey{(I)}{label=\rm{(\Roman*)}}
\SetEnumitemKey{(a)}{label=\rm{(\alph*)}}
\SetEnumitemKey{(A)}{label=\rm{(\Alph*)}}

\usepackage[protrusion=true,expansion=true]{microtype}

\usepackage[%backref=page,
            colorlinks=true,
            linkcolor=blue]{hyperref}

\usepackage[capitalise]{cleveref}

\usepackage[backend = bibtex,
            language = english ,
            style    = numeric ,
            firstinits = true,
            isbn = false,
            url = false,
            doi = false,
            sorting = nyt,
            ]{biblatex}
\DeclareNameAlias{sortname}{last-first}
\renewbibmacro{in:}{%
  \ifentrytype{article}{}{%
  \printtext{\bibstring{in}\intitlepunct}}}
\bibliography{Mi_biblioteca}
\AtEveryBibitem{\clearlist{language}} % clears language
\AtEveryBibitem{\clearfield{note}}    % clears notes

\newtheorem{thm}{Theorem}
\newtheorem*{thm*}{Theorem}
\newtheorem{cor}[thm]{Corollary}
\newtheorem{lem}[thm]{Lemma}
\newtheorem*{lem*}{Lemma}

\theoremstyle{definition}

\theoremstyle{remark}

\newcommand{\defin}[1]{\emph{#1}}

\newcommand{\comment}[1]{{}}

\newcommand\NN{\mathbb{N}}
\newcommand\ZZ{\mathbb{Z}}

\newcommand{\Imp}{\ \Rightarrow \ }              % => (implica doble)
\newcommand{\Biimp}{\ \Leftrightarrow \ }           % <=>
\newcommand{\st}{\ \colon \,}                       %such that

\newcommand{\Edges}{E}
\newcommand{\graphstyle}[1]{\mathsf{#1}}

\newcommand{\Graphi}{\graphstyle{\Gamma}}          % graphs
\newcommand{\Graphii}{\graphstyle{\Lambda}}

\newcommand{\pathi}{\graphstyle{P}}
\newcommand{\cyclei}{\graphstyle{C}}
\newcommand{\Groupi}{G}          % groups
\newcommand{\groupi}{g}

\newcommand{\Geni}{X}
\newcommand{\geni}{x}
\newcommand{\Genii}{Y}
\newcommand{\genii}{y}
\newcommand{\Reli}{R}

\newcommand{\pres}[2]{\langle \, #1 \! \mid \! #2 \, \rangle}

\newcommand{\Pcg}[1]{\left\langle \, #1 \, \right\rangle}    %{G{#1}}
\newcommand{\End}{\operatorname{End}}

\newcommand{\Fix}{\operatorname{Fix}}
\newcommand{\Per}{\operatorname{Per}}

\newcommand{\trivial}{1}     % trivial subgroup, alt: \{1\} , 1
\newcommand{\morphi}{\varphi}

\title{Some characterizations of Howson PC-groups}

\author{
\textbf{Jordi Delgado}\\[3pt]
Dept. Mat. Apl. III,\\
Universitat Polit\`ecnica de Catalunya,\\
Barcelona\\
\texttt{jorge.delgado@upc.edu}
}

\begin{document}

%\begin{abstract}
%We show that in the class of partially-commutative groups the conditions of
%\begin{enumerate*}[(i)]
%  \item[] being Howson,
%  \item[] being fully residually free, and
%  \item[] being a free product of free-abelian groups,
%\end{enumerate*}
%are equivalent.
%\end{abstract}

%\label{top}

\maketitle

{\small

}

\begin{abstract}
We show that in the class of partially commutative groups, the conditions of
%a group
being Howson, being fully residually free, and being free product of free-abelian groups, are equivalent.
\end{abstract}

\bigskip

In \cite{rodaro_fixed_2013}, the authors study the family of finitely generated partially commutative groups for which
%every endomorphism has a finitely generated fixed point subgroup.
the fixed points subgroup of every endomorphism is finitely generated. Concretely,
%in Theorem 3.1
they characterize this family as those groups consisting in (finite) free products of finitely generated free-abelian groups.

%Although it is an immediate consequence of well know facts,
In this note we provide an elementary proof for two extra characterizations of this family, namely: being Howson,
and  being a limit group.
%(i.e.~a finitely generated fully residually free group).
Moreover, we observe that, for some of the properties, no restriction in the cardinal of the generating set is needed, and the result holds in full generality (i.e.~for every --- possibly infinitely generated --- partially commutative group).

%\medskip
%
%A brief description of the notions involved in this characterizations follows.

\medskip

%Let $X$ be an arbitrary (possibly infinite) undirected graph. The \defin{partially commutative} group
%(\defin{PC-group}, for short) $G(X)$ is the group generated by the set of vertices $VX$, with a relation of
%commutativity $[u,v]=1$ for each pair of adjacent \note{incident $\rightarrow$\\ adjacent} vertices $(u,v)\in EX$. This way, the PC-group corresponding to a graph with no edges is a free group, and that corresponding to a complete graph is a free-abelian group (in both cases, of rank equal to the number of
%vertices).

\section{Preliminaries}

We call \defin{partially commutative groups} (\defin{PC-groups}, for short)
the groups that admit a presentation in which all the relations are commutators between generators, i.e.~a
presentation of the form
$
\pres{\Geni}{\Reli}
$,
where $\Reli$ is a subset of $[\Geni,\Geni]$ (the \emph{set} of commutators between elements in~$\Geni$).
\goodbreak

We can represent this situation in a very natural way through the
%(undirected)
(simple) graph~$\Graphi = (\Geni,\Edges)$ having
as vertices the generators in $\Geni$, and two vertices $\geni, \genii \in \Geni$ being adjacent if and only if its
commutator $[\geni, \genii]$ belongs to $\Reli$; then we say that the PC-group is \defin{presented} by the graph $\Graphi$, and we denote it by $\Pcg{\Graphi}$.
Recall that a \defin{simple graph} is undirected, loopless, and without multiple edges; so,~$\Graphi$ is nothing more than a symmetric and irreflexive binary relation in $X$.

%The \defin{full subgraph} of a graph $\Graphi = (\Geni,\Edges)$ spanned by a subset of vertices $\Genii \subseteq \Geni$ is the graph with vertex set $\Genii$ having exactly the edges that appear in~$\Graphi$ over.
A subgraph of a graph $\Graphi = (\Geni,\Edges)$ is said to be \defin{full} if it has exactly the edges that appear in~$\Graphi$ over the same vertex set, say $\Genii \subseteq \Geni$.
Then, it is called the \defin{full subgraph of $\Graphi$ spanned by $\Genii$}, and we denote it by~$\Graphi[Y]$.
If $\Graphi$ has a full subgraph isomorphic to a certain graph~$\Graphii$, we will abuse the
terminology and say that $\Graphii$ is (or appears as) a full subgraph of $\Graphi$; we denote this situation by~$\Graphii \leqslant \Graphi$. When
%a graph $\Graphi$ does not have any full subgraph isomorphic to a graph $\Graphii$, we say that $\Graphi$ is \defin{$\Graphii$-free}.
%a graph $\Graphii$ does no appear as a full subgraph of a graph $\Graphi$, we say that $\Graphi$ is \defin{$\Graphii$-free}.
none of the graphs belonging to a certain family $\mathcal{F}$ appear as a full subgraph of $\Graphi$, we say that $\Graphi$ is \defin{$\mathcal{F}$-free}. In particular, a graph $\Graphi$ is \defin{$\Graphii$-free} if it does not have any full subgraph isomorphic to $\Graphii$.

It is clear that every graph $\Graphi$ presents exactly
one PC-group; that is, we have a surjective map~$\Graphi \mapsto \Pcg{\Graphi}$
between (isomorphic classes of) simple graphs and (isomorphic classes of) PC-groups.
A key result proved by Droms in \cite{droms_isomorphisms_1987} states that
%\R{if the graph $\Graphi$ is finite}, then
%the map $\Graphi \mapsto \Pcg{\Graphi}$
this map
is, in fact, bijective.
Therefore, we have an absolutely transparent geometric characterization of isomorphic classes of
%\R{finitely generated}
PC-groups: we can identify them with simple graphs.

%So, some graph related notions will be relevant for our discussion.
%Two types of operations between graphs represent a prominent role in this identification, namely the disjoint union of graphs, and the graph join. Recall that, given two graphs $\Graphi = (\Geni,\Edges)$ and $\Graphi' = (\Geni',\Edges')$, the \defin{disjoint union} of $\Graphi$ and $\Graphi'$ is the graph
%%${\Graphi \sqcup \Graphi'} := (\Geni \sqcup \Geni , \Edges \sqcup \Edges')$
%${\Graphi \sqcup \Graphi'}$ with vertex set $\Geni \sqcup \Geni$ and edge set $\Edges \sqcup \Edges'$, while the \defin{join graph} of $\Graphi$ and $\Graphi'$ is the graph
%%${\Graphi \sqcup \Graphi'} := (\Geni \sqcup \Geni , \Edges \sqcup \Edges')$
%${\Graphi \join \Graphi'}$ obtained by adding to ${\Graphi \sqcup \Graphi'}$ the edges joining every vertex of $\Graphi$ to every vertex of $\Graphi'$,

This way, the PC-group corresponding to a graph with no edges is a free group, and the one corresponding to a complete graph is a free-abelian group (in both cases, with rank equal to the number of vertices). So, we can think of PC-groups as a generalization of these two extreme cases including all the intermediate commutativity situations between them.

Similarly, disjoint unions and \defin{joins of graphs}
%(recall that the \defin{graph join} of $\{\Graphi_i\}$ is the graph\comment{$\bigvee_{\!i} \Graphi_i$} obtained by adding to the disjoint union\comment{ $\bigsqcup_i \Graphi_i$} of the $\Graphi\!_i$'s  all the possible edges between any pair of distinct $\Graphi\!_i$'s)
(i.e.~disjoint unions with all possible edges between distinct constituents\comment{$\Graphi\!_i$'s} added)
correspond to free products and weak direct products of PC-groups respectively. So, for example, the finitely generated free-abelian times free group $\ZZ^m \times F_n$ is presented by the join of a complete graph of order $m$ and an edgeless graph of order~$n$.

All these facts are direct from definitions, and make the equivalence between the conditions in the following lemma almost immediate as well.

\begin{lem} \label{lem:graph equivalence}
Let $\Graphi$ be an arbitrary simple graph, and $\Pcg{\Graphi}$ the corresponding
PC-group. Then, the following conditions are equivalent:
\begin{enumerate}[(i)]
\item\label{ite:P3-free} the path on three vertices $\pathi_{3}$ is not a full subgraph of $\Graphi$ (i.e.~$\Graphi$ is $\pathi_{3}$-free),
\item\label{item:transitive} the reflexive closure of $\Graphi $ is a transitive binary relation,
\item\label{item:sqcup of complete} $\Graphi$ is a disjoint union of complete graphs,
\item\label{item:free product of free-abelian} $\Pcg{\Graphi}$ is a free product of free-abelian groups. \qed
\end{enumerate}
\end{lem}

\goodbreak

The next lemma, for which we provide an elementary proof, is also well known. We will use it in the proof of \cref{thm:characterization}.
\begin{lem} \label{lem:induced PC-group}
Let $\Graphi$ be an arbitrary simple graph, and $Y$ a subset of vertices of $\Graphi$. Then, the subgroup of $\Pcg{\Graphi}$ generated by $Y$ is
isomorphic to the PC-group presented by
%$\Pcg{\graphind{Y}{\Graphi}}$
 $\Graphi[Y]$.
%the full subgraph of $\Graphi$ spanned by $Y$.
\end{lem}

\begin{proof}
Let $\Geni$ be the set of vertices of $\Graphi$
(then $\Genii \subseteq \Geni$),
%Denote $\Graphi[Y]$ the the full subgraph of $\Graphi$ spanned by $Y$,
and consider the following two homomorphisms:
\begin{equation*}
\begin{array}{rcccc}
\Pcg{\Graphi[Y]} & \overset{\alpha}\longrightarrow &\Pcg{\Graphi}\\
\genii & \longmapsto     & \ \ \genii  \\
\phantom{x}
\end{array}
\
\begin{array}{c}
\text{,}\\
\phantom{y}\\
\phantom{x}
\end{array}
\
\begin{array}{rcccc}
\Pcg{\Graphi} & \overset{\rho}{\longrightarrow} &\Pcg{\Graphi[Y]} \\
\Genii \ni \genii         & \longmapsto     & \genii \\
\Geni \setminus \Genii \ni \geni         & \longmapsto     & \trivial
\end{array}
\hspace{-7pt}
\begin{array}{c}
\text{.}\\
\phantom{y}\\
\phantom{x}
\end{array}
\end{equation*}
It is clear that both $\alpha$ and $\rho$ are well defined homomorphisms (they obviously respect relations).
%the first one being surjective, and the second one being injective.
Moreover, note that the composition $\alpha \rho$ ($\alpha$ followed by $\rho$) is the identity map on $\Pcg{\Graphi[Y]}$. Therefore, $\alpha$ is a monomorphism, and thus $\Pcg{\Graphi[Y]}$ is isomorphic to its image under $\alpha$, which is exactly the subgroup of $\Pcg{\Graphi}$ generated by $Y$, as we wanted to prove.
\end{proof}

\medskip
%Now we present a couple of group properties (one, in general, stronger than the other) that (as we will see in \cref{thm:characterization}), in the case of partially commutative groups, turn out to be equivalent.

A group
%$G$
is said to satisfy the \defin{Howson property} (or to be \defin{Howson}, for short) if
%every finite intersection of finitely generated subgroups of $G$ is again finitely generated.
the intersection of any two finitely generated subgroups is again finitely generated.
It is well known that free and free-abelian groups are Howson (see, for example,~\cite{bogopolski_introduction_2008}~and~\cite{hungerford_algebra_1974} respectively).

However, not every PC-group is Howson: for example, the class of free-abelian times free groups
%(generated by the join of a complete and an edgeless graph),
(studied in~\cite{delgado_algorithmic_2013})
turns out to be not Howson in every non-degenerate case (i.e.~they are Howson if and only if they do not have $\ZZ \times F_2$ as a subgroup). So, it is a natural question to ask for a characterization of Howson PC-groups, and we will see in \cref{thm:characterization} that the very same condition (not containing $\ZZ \times F_2$ as a subgroup) works for a general PC-group.

For limit groups there are lots of different equivalent definitions. We shall use the one using fully residually freeness (see~\cite{wilton_introduction_2005} for details):~a group $\Groupi$ is \defin{fully residually free} if for every finite subset $S \subseteq \Groupi$ such that $1 \notin S$, there exist an  homomorphism $\morphi$ from $\Groupi$ to a free group such that $1 \notin \morphi (S)$.
%(note that, in this definition one can assume the free group to be finitely generated).
%if every finite family of non-trivial elements of $\Groupi$ ``fully survives'' in some free quotient of $\Groupi$.
Then, a \defin{limit group} is a finitely generated fully residually free group.
From this definition, it is not difficult to see that both free and free-abelian groups are fully residually free, and that subgroups and free products of fully residually free groups are again fully residually free.

%Another easy application of the definition is the fact that limit groups are \emph{commutative-transitive} i.e., for $1\neq x,y,z\in G$, if $x$ commutes with $y$ and $y$ commutes with $z$ then $x$ commutes with $z$ (see~\cite{wilton_introduction_2005}). \note{POTSER NO CAL}

\section{Characterizations
%of $\pathi_{3}$-free PC-groups
}

As proved by Rodaro, Silva, and Sykiotis (Theorem 3.1 in \cite{rodaro_fixed_2013}),
if we restrict to finitely generated PC-groups, \cref{lem:graph equivalence} describes exactly the family of those
having finitely generated fixed point subgroup for every endomorphism
(or equivalently, those
having finitely generated periodic point subgroup for every endomorphism).

In the following theorem, we provide two extra characterizations for the PC-groups described in \cref{lem:graph equivalence}  (including the infinitely generated case).
For completeness in the description, we summarize them in a single statement together with the conditions
discussed above.
%from Theorem 3.1 in \cite{rodaro_fixed_2013}.

\begin{thm} \label{thm:characterization}
Let $\Graphi$ be an arbitrary (possibly infinite) simple graph, and $\Pcg{\Graphi}$ the
PC-group presented by $\Graphi$. Then, the following conditions are equivalent:
\begin{enumerate}[(i)]
\item\label{item:fully residually free} $\Pcg{\Graphi}$ is fully residually free,
\item\label{item:Howson} $\Pcg{\Graphi}$ is Howson,
\item\label{item:F2 x Z not in G} $\Pcg{\Graphi}$ does not contain $\ZZ \times F_2$ as a subgroup,
%\item\label{ite:P3-free} $\Graphi$ is $\pathi_{3}$-free,
%\item\label{item:transitive} the reflexive closure of $\Graphi $ is transitive (as a binary relation),
%\item\label{item:sqcup of complete} $\Graphi$ is the disjoint union of complete graphs,
\item\label{item:free product of free-abelian} $\Pcg{\Graphi}$ is a free product of free-abelian groups.
\end{enumerate}
Moreover, if $\Graphi$ is finite, then the following additional conditions are also equivalent:
\begin{enumerate}[resume,(i)]
%\item\label{item:limit group} $\Groupi$ is a limit group,
\item\label{item:Fix fg} For every $\morphi \in \End{\Pcg{\Graphi}}$, the subgroup
\begin{equation*}
\Fix \morphi = \{ \groupi \in \Pcg{\Graphi} \st \morphi(\groupi) = \groupi \}
\end{equation*}
of fixed points of $\morphi$ is finitely generated.
\item\label{item:Per fg} For every $\morphi \in \End{\Pcg{\Graphi}}$, the subgroup
\begin{equation*}
\Per \morphi = \{ \groupi \in \Pcg{\Graphi} \st \exists n\geq 1 \ \morphi^{n}(\groupi) = \groupi \}
\end{equation*}
of periodic points of $\morphi$ is finitely generated.
\end{enumerate}
\end{thm}

\begin{proof}
$[\ref{item:fully residually free}\,\Rightarrow \, \ref{item:Howson}]$. Dahmani obtained this result for limit groups
(i.e.~assuming~$\Pcg{\Graphi}$~finitely generated)
as a consequence of them being hyperbolic relative to their maximal abelian non-cyclic subgroups (see Corollary~0.4 in~\cite{dahmani_combination_2003}). We note that the finitely generated condition is superfluous for this implication since the Howson property involves only finitely generated subgroups, and every subgroup of a fully residually free group is again fully residually free.

$[\ref{item:Howson}\, \Rightarrow \, \ref{item:F2 x Z not in G}]$. It is enough to prove that the group $\ZZ \times F_2$ does not satisfy the Howson
property. The following argument is described as a solution to exercise 23.8(3) in~\cite{bogopolski_introduction_2008} (see also \cite{delgado_algorithmic_2013}). Indeed, if we write $\ZZ \times F_2 = \pres{t}{-} \times \pres{a,b}{-} $, then the subgroups
%$H=\langle a, b\rangle \leqslant \ZZ \times F_2$ and
 \begin{equation*}
 \begin{array}{l}
H=\langle a, b\rangle = F_2  \leqslant \ZZ \times F_2 \text{, and}\\[2pt]
K=\langle ta,b\rangle =\{ w(ta,b) \mid w\in F_2 \}=\{ t^{|w|_a}w(a,b) \mid w\in F_2 \}\leqslant \ZZ \times F_2
\end{array}
 \end{equation*}
are both finitely generated, but its intersection
 \begin{equation*}
H\cap K =\{ t^0 w(a,b) \mid w\in F_2,\, |w|_a=0 \} =
%\llangle b\rrangle_{F_2}=
\langle\hspace{-2.5pt}\langle b\rangle\hspace{-2.5pt}\rangle_{F_2} =\langle a^{-k}ba^k,\, k\in \mathbb{Z}\rangle
 \end{equation*}
is infinitely generated, as you can see immediately from its Stallings graph~(see~\cite{stallings_topology_1983} and~\cite{kapovich_stallings_2002})
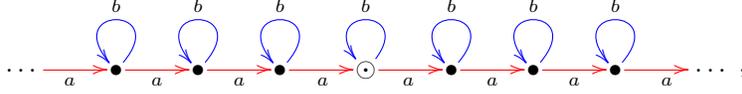
\vspace{-5pt}
\begin{equation*}
\xymatrix
@1
{
\raisebox{2pt}{\ldots} \ar@[red][r]_{a} & \bullet \ar@[blue]@(ru,lu)[]_{b} \ar@[red][r]_{a} & \bullet \ar@[blue]@(ru,lu)[]_{b} \ar@[red][r]_{a} & \bullet
\ar@[blue]@(ru,lu)[]_{b} \ar@[red][r]_{a} &
\odot \ar@[blue]@(ru,lu)[]_{b} \ar@[red][r]_{a} & \bullet \ar@[blue]@(ru,lu)[]_{b} \ar@[red][r]_{a} & \bullet \ar@[blue]@(ru,lu)[]_{b} \ar@[red][r]_{a} & \bullet
%\ar@[blue]@(ru,lu)[]_{b} \ar@[red][r]_{a} & \bullet
\ar@[blue]@(ru,lu)[]_{b} \ar@[red][r]_{a} & \, \raisebox{2pt}{\ldots \ ,}
}
\end{equation*}
or using this alternative argument: Suppose that $H \cap K$ is finitely generated, then there exist an $m \in \NN$ such that $a^{m+1} b a^{-(m+1)} \in \langle a^{-k}ba^k,\, k\in [-m,m]\rangle$, and thus $a^{m+1}$ equals the reduced form of some prefix of $w(a^m b a^{-m},\ldots,b,\ldots,a^{-m} b a^{m})$, for some word $w$. However, the sum of the exponents of $a$ in any such prefix must be in $[-m,m]$, which is a contradiction.

Note that both $H$
and $K$ are free groups of rank two whose intersection is infinitely generated. This fact, far from violating the Howson property, means that they are not simultaneously contained in any free subgroup of $\ZZ \times F_2$.

$[\ref{item:F2 x Z not in G}\,\Rightarrow \, \ref{item:free product of free-abelian}]$.
%It is not difficult to see that the subgroup of $\Pcg{\Graphi}$ generated by a subset $Y$ of vertices  of $\Graphi$ is isomorphic to the PC-group
%%$\Pcg{\graphind{Y}{\Graphi}}$
%presented by the full subgraph of $\Graphi$ spanned by $Y$.
From \cref{lem:induced PC-group},
if $\Pcg{\Graphi}$ does not contain the group $\ZZ \times F_2$ (which is presented by $\pathi_3$) as a subgroup, then $\pathi_3$ is not a full subgraph of $\Graphi$. Equivalently (see \cref{lem:graph equivalence}), $\Pcg{\Graphi}$ is a free product of free-abelian groups.

$[\ref{item:free product of free-abelian} \,\Rightarrow \, \ref{item:fully residually free}]$. This is again clear, since both free-abelian groups and free products of fully residually free groups are again fully residually free. Note here, that no cardinal restriction is needed; neither for the rank of the free-abelian groups, nor for the number of factors in the free product, since the definition of fully residually freeness involves only finite families.

Finally, for the equivalence between \ref{item:free product of free-abelian}, \ref{item:Fix fg} and \ref{item:Per fg} under the finite generation hypothesis, see Theorem 3.1 in \cite{rodaro_fixed_2013}.
\end{proof}

Observe that an immediate corollary of \cref{lem:induced PC-group} is that
the PC-group presented by any full subgraph $\Graphii \leqslant \Graphi$ is itself a subgroup of the PC-group presented by $\Graphi$, i.e.~for every pair of graphs $\Graphi,\Graphii$,
\begin{equation} \label{eq:full subgraph -> subgroup}
\Graphii \leqslant \Graphi \Imp \Pcg{\Graphii} \leqslant \Pcg{\Graphi}.
\end{equation}
%Then, we will say that $\Pcg{\Graphii}$ is visible.
This property provides a distinguished family of subgroups
(which we will call visible)
of any given PC-group.
More precisely, we will say that a PC-group $\Pcg{\Graphii}$ is a \defin{visible subgroup of} a PC-group~$\Pcg{\Graphi}$ --- or that $\Pcg{\Graphii}$ is \defin{visible in}  $\Pcg{\Graphi}$ --- if~$\Graphii$~appears as a full subgraph of~$\Graphi$.
%Otherwise we will say that $\Pcg{\Graphii}$ is an \defin{invisible subgroup of}  $\Pcg{\Graphi}$, or that $\Pcg{\Graphii}$ is \defin{invisible in}  $\Pcg{\Graphi}$.

%Note that
%the converse of \eqref{eq:full subgraph -> subgroup} is not true in general, i.e.~not every partially commutative subgroup
%%$\Pcg{\Graphii}$
%of a PC-group $\Pcg{\Graphi}$ needs to be presented by a full subgraph of $\Graphi$ (consider, for example, $F_3$ as a subgroup of~$F_2$).

%\begin{defn}
%A subgroup of a PC-group $\Pcg{\Graphi}$ is \defin{visible}
%%(in $\Graphi$)
%if it is presented by certain full subgraph of $\Graphi$.
%\end{defn}
Of course, visible subgroups
%(of PC-groups)
are PC-groups as well,  but
%(recall the $F_3 \leqslant F_2$ example)
not every partially commutative subgroup of a PC-group is visible
(for example, $F_3$ is obviously not visible in~$F_2$).
%So, the family of visible subgroups of a given PC-group is, in general, a proper family of subgroups.

%We recall here, that a subgroup of a PC-group does not even have to be again a PC-group: Droms, in \cite{droms_subgroups_1987}, characterized the finitely generated PC-groups
%%whose subgroups are all again PC-groups
%with all their subgroups being again PC-groups
%as precisely those presented by a~$\{\pathi_4,\cyclei_4\}$-free graph.

Note that although ``visibility'' is a relative property (a PC-group can be visible in a certain group, but not in another one), there exist PC-groups
which are visible in every PC-group in which they appear as a subgroup; we will call them explicit.
That is, a given PC-group $\Pcg{\Graphii}$ (or the graph $\Graphii$ presenting it) is \defin{explicit}\ if for every graph $\Graphi$, %the following equivalence holds
\begin{equation} \label{eq:full subgraph <-> subgroup}
\Graphii \leqslant \Graphi \Biimp \Pcg{\Graphii} \leqslant \Pcg{\Graphi}.
\end{equation}

%such that, every time they appear as a subgroup of a Pc-group $\Pcg{\Graphi}$ if they  which are not only visible as subgroups of certain PC-groups but always visible as subgroups of PC-groups.
%So, the converse of \eqref{eq:full subgraph -> subgroup} is not true in general,
%but it can be shown to be true for certain particular PC-groups~$\Pcg{\Graphii}$ (or equivalently, for certain graphs $\Graphii$) which we will call \defin{explicit} (those who are always visible as subgroups of PC-groups).

%\begin{defn}
%A graph $\Graphii$
%(or the PC-group $\Pcg{\Graphii}$ presented by it)
%is \defin{explicit},
%if whenever $\Pcg{\Graphii}$ appears as a subgroup of a PC-group~$\Pcg{\Graphi}$, $\Graphii$ also appears as a full subgraph of $\Graphi$; i.e.~if for every graph $\Graphi$,
%%$\Graphii \leqslant \Graphi \Biimp \Pcg{\Graphii} \leqslant \Pcg{\Graphi}$.
%\begin{equation*} \label{eq:full subgraph <-> subgroup}
%\Graphii \leqslant \Graphi \Biimp \Pcg{\Graphii} \leqslant \Pcg{\Graphi}.
%\end{equation*}
%\end{defn}
For example, it is straightforward to see that the only explicit edgeless graphs are the ones with zero, one, and two vertices: the first two cases are obvious, and for the third one, note that if $F_2 \leqslant \Groupi$ then $\Groupi$ can not be abelian. Finally, for $n\geq3$, it is sufficient to note (again) that $F_n$ is not a visible subgroup of~$F_2$.

At the opposite extreme, a well-known result (Lemma 18 in \cite{kim_embedability_2013}) states that the maximum rank of a free-abelian subgroup of a f.g.~PC-group
$\Pcg{\Graphi}$
is the size of a largest complete subgraph in $\Graphi$. An immediate corollary is that every (finite) complete graph is explicit.

In the last years,
%such questions about
embedability between PC-groups has been matter of growing interest and research (see \cite{kambites_commuting_2009}, \cite{kim_embedability_2013} and \cite{casals-ruiz_embedddings_2013}) which has provided some new examples of explicit graphs, such as the square $\cyclei_4$ (proved by Kambites, in \cite{kambites_commuting_2009}), or the path on four vertices $\pathi_4$ (proved by Kim and Koberda, in \cite{kim_embedability_2013}).

To end with, we just remark that our characterization theorem (\cref{thm:characterization}) immediately provides a new member of this family.% is that the path on three vertices $\pathi_3$ is also explicit.
%We use the characterization to reformulate this result in a more algebraic way.

\begin{cor}
The path on three vertices $\pathi_3$ is explicit. \qed
\end{cor}

%%%%%%%%%%%%%%%%%%%%%%% BIBLIOGRAPHY & END %%%%%%%%%%%%%%%%%%%%%%
\renewcommand{\bibfont}{\normalfont\small}
\bibitemsep = 1ex
\printbibliography

\subsection*{Acknowledgement}
The author gratefully acknowledges the support of \emph{Universitat Polit\`{e}cnica de Catalunya} through PhD grant number 81--727, and the partial support from the MEC (Spanish Government) through research project number MTM2011-25955. I would also like to express my gratitude to Enric Ventura for his constant support, and his insightful comments and suggestions.

%\section*{} \label{authors}
%
%\begin{minipage}[t]{0.46\textwidth}
%\textbf{Jordi Delgado Rodr\'iguez}\hyperref[top]{$^{*}$}
%
%\emph{Dept. Mat. Apl. III,}
%
%\emph{Universitat Polit\`ecnica de Catalunya,}
%
%\emph{Manresa, Barcelona.}
%
%%\emph{email}:
%%\begin{tabular}{l}
%%\texttt{jorge.delgado@upc.edu}\\
%%\texttt{jdelgado.upc@gmail.com}
%%\end{tabular}
%\texttt{jorge.delgado@upc.edu}
%%\texttt{jdelgado.upc@gmail.com}
%\end{minipage}

\end{document}